\documentclass[11pt,a4paper]{article}
\usepackage{amsmath,amsfonts,amssymb,amsthm,makeidx}
\input xy  \xyoption{all}

\theoremstyle{plain}
\newtheorem{teo}{Theorem}[section]
\newtheorem{prop}[teo]{Proposition}
\newtheorem{pro}[teo]{Problem}

\newtheorem{cor}[teo]{Corollary}
\newtheorem{corollary}[teo]{Corollary}
\newtheorem*{nteo*}{\namedthmname}
\newtheorem{que}[teo]{Question}

\newtheorem*{ncor*}{\namedcorname}

\theoremstyle{definition}

\newtheorem{defin}[teo]{Definition}
\newtheorem{fact}[teo]{Fact}
\newtheorem{rem}[teo]{Remark}
\newtheorem{exa}[teo]{Example}
\newtheorem{example}[teo]{Example}

\theoremstyle{remark}

\newtheorem{notation}[teo]{Notation}

\DeclareMathOperator{\G}{\mathfrak{G}_\delta}
\DeclareMathOperator{\GG}{\mathfrak{G}_{\delta\sigma}}
\DeclareMathOperator{\F}{\mathfrak{F}_\sigma}
\DeclareMathOperator{\sLC}{\mathfrak{LC}}
\DeclareMathOperator{\SG}{\mathfrak{G}_\delta}
\DeclareMathOperator{\SGG}{\mathfrak{G}_{\delta\sigma}}
\DeclareMathOperator{\FF}{\mathfrak{F}_{\sigma\delta}}
\DeclareMathOperator{\SF}{\mathfrak{F}_\sigma}
\DeclareMathOperator{\SFF}{\mathfrak{F}_{\sigma\delta}}
\DeclareMathOperator{\sPol}{\mathfrak{Pol}}
\DeclareMathOperator{\CH}{\mathfrak{Char}}
\DeclareMathOperator{\ACH}{\mathfrak{a-Char}}
\DeclareMathOperator{\tu}{t_{\us}}

\DeclareMathOperator{\sv}{s_{\vs}}
\DeclareMathOperator{\su}{s_{\us}}

\newcommand{\vvv}{\mathbf{v}}
\newcommand{\uuu}{\mathbf{u}}
\newcommand{\J}{\mathbb{J}}

\newcommand{\T}{\mathbb{T}}
\newcommand{\PP}{\mathbb{P}}
\newcommand{\Z}{\mathbb{Z}}
\newcommand{\Q}{\mathbb{Q}}
\newcommand{\R}{\mathbb{R}}
\newcommand{\N}{\mathbb{N}}
\newcommand{\us}{\mathbf{u}}
\newcommand{\cc}{\mathfrak{c}}
\newcommand{\vs}{\mathbf{v}}
\newcommand{\ep}{\varepsilon}

\newcommand{\tTu}{\tu(\T)}

\newcommand{\lb}{\left\|}
\newcommand{\rb}{\right\|}

\title{Questions on the Borel Complexity of Characterized Subgroups}
\author{Dikran Dikranjan and Daniele Impieri}

\begin{document}

\maketitle
\begin{abstract}
A subgroup $H$ of a compact abelian group $X$ is said to be characterized by a sequence $(v_n)$ of characters of $X$, if  $H =  \left\{x\in X:v_n(x)\to0 \text{ in }\T\right\}$. We discuss  various  known facts and open problems about the Borel complexity of characterized subgroups of compact abelian groups.
\end{abstract}

  \section{Introduction}
  In the sequel  $\T=\R/\Z$ will be denote the circle group written additively, identifying each $x\in \T$ with an obviously determined element of the interval $[0,1)$. In these terms we consider the norm in $\T$ defined by $\|x\| = \min\{x, 1-x\}$ for $x\in \T$. It defines an invariant metric $d$ on $\T$ by letting $d(x,y) = \|x-y\|$, for $x,y\in \T$. 
  
For a topological abelian group $X$, a continuous homomorphism $\chi:X\to\T$ is called a \emph{character} of $X$. Denote by $\widehat{X}$ the group of all characters of $X$, that is the Pontryagin dual of $X$.
	\begin{defin}
	 For a compact abelian group $X$ and  a sequence $\vs=(v_n)_{n\in\N}$ of characters of $X$ let 
	  \begin{equation*}
	  \sv(X):=\left\{x\in X:v_n(x)\to0 \text{ in }\T\right\}.
	  \end{equation*}
	\end{defin}

One can easily check that  $  \sv(X)$ is a subgroup of $X$. 
	
	\begin{defin}\label{DefChar}  {\rm \cite{DMT}} A subgroup $H$ of a compact abelian group $X$ is called \emph{characterized} if there exists a sequence of characters $\vs$ such that $H=\sv(X)$. We also say that $\vs$ \emph{characterizes} $H$ and denote by $\CH(X)$ the family of all characterized subgroups of $X$.
	\end{defin}

The following general problem has been studied in \cite{DG}.

\begin{pro}\label{probBor}
Describe the Borel complexity of $\sv(X)$ for a compact abelian group $X$.
\end{pro}

The Borel complexity of a Borel set is the \lq\lq lowest\rq\rq  class in the Borel Hierarchy where such a set belongs to. Let us recall the first six classes of the Borel Hierarchy.
 
 \begin{notation}\label{nsub}
  
  \begin{equation*}
	\begin{matrix}
		& \{\text{open sets}\} \!&\!\subseteq\! &\! \{G_\delta\text{-sets}\} \!&\!\subseteq \!&\!\{\text{$G_{\delta\sigma}$-sets}\}&\\
		& \{\text{closed sets}\} \!&\! \subseteq \!& \!\{F_\sigma\text{-sets}\}&\subseteq \!&\!\{\text{$F_{\sigma\delta}$-sets}\}&
	\end{matrix},
\end{equation*}
where $G_\delta$-sets are countable intersection of open sets, $F_\sigma$-sets are countable union of closed sets, $G_{\delta\sigma}$ are countable union of $G_\delta$-sets and $F_{\sigma\delta}$-sets are countable intersection of $F_\sigma$-sets.
 
 In this paper we are interested in Borel subgroups, hence we shall denote by $\G(X)$, $\F(X)$, $\GG(X)$ and $\FF(X)$ the class, respectively, of $G_{\delta}$-subgroups, $F_{\sigma}$-subgroups, $G_{\delta\sigma}$-subgroups and $F_{\delta\sigma}$-subgroups of a given topological group $X$ (e.g., accordingly, an $F_{\sigma\delta}$-subgroup means a subgroup that is an $F_{\sigma\delta}$-set as a subset, etc.). 
\end{notation}

\begin{rem}\label{rem1} Every characterized subgroup $H$  of a compact abelian group $X$ is an $F_{\sigma\delta}$-subgroup of $X$, and hence $H$ is a Borel subset of $X$. Indeed, if $H$ is characterized by a sequence $\vs$, this fact directly follows from the equality
	  \begin{equation*}
	  \sv(X) = \bigcap_{M=1}^\infty \bigcup_{m=0}^\infty \left( \bigcap_{n\geq m}^\infty \left\{ x\in X : \lb v_n(x)\rb \leq \frac{1}{M} \right\} \right).
	\end{equation*}
\end{rem}

Since the characterized subgroups are Borel, if $H\in\CH(X)$ then either $|H|\le\aleph_0$ or $|H|=\cc$ by Alexandroff-Hausdorff's theorem \cite[\S 37, Theorem 3]{KK}. 

One can define characterized subgroups of an arbitrary abelian topological group. Similarly as in Remark \ref{rem1}, one can see
that they are always Borel subgroups (actually,  $F_{\delta\sigma}$-subgroups). On the other hand, Borel subgroups of the uncountable Polish abelian groups need not be characterized in general. Indeed, the above remark establishes an upper bound for the Borel complexity of characterized subgroups.
On the other hand, every uncountable Polish group contains arbitrarily complicated Borel subgroups \cite[Theorem 2.1]{FS}.
This implies that every uncountable Polish group contains a Borel subgroup that is not characterized. 

The main issue discussed in this paper is when characterized subgroups are $F_\sigma$, as well as the opposite 
question: when an $F_\sigma$-subgroup of a compact abelian group is a characterized subgroup.  
	
The paper is organized as follows. Section 2 is dedicated to the general case of arbitrary compact abelian groups. In \S 2.1 we recall some known facts
about characterized subgroups and their Borel complexity, starting from the case of countable subgroups and $G_\delta$-subgroups. Then comes the reduction
to the case of metrizable compact abelian groups and some facts related to $F_\sigma$-subgroups. This introduces in a natural way the next topic, in  \S 2.2, 
of the next sharper necessary condition for a subgroup to be characterized, namely, to be Polishable (beyond being a Borel subgroup). 
This gives rise to a finer Polish topology $\tau(H)$ (witnessing Polishability) for  a characterized subgroup $H$ of a compact metrizable abelian group. Further properties of this topology are discussed in \S 2.3, local compactness being the prominent among all. These stronger properties of $\tau(H)$ turn out to be sufficient for a group $H$ to be characterized. In \S 2.4 comes a general criterion for a characterized subgroup $H$ to be an $F_\sigma$-subgroup in terms of a new topology, obtained by a suitable modification of  $\tau(H)$. 

Section 3 is dedicated to the case of subgroups of $\T$. Since $\widehat \T = \Z$, now the sequences of characters are simply sequences of integers. 
It turns out that the proper characterized subgroups of $\T$ whose associated Polish
topology is locally compact are precisely the countable ones. Moreover, among all compact abelian groups only the group $\T$ has this property. 
A special attention is dedicated to sequences $\uuu$ such that $u_n | u_{n+1}$. In this case we obtain a complete description of the characterized subgroups of $\T$ that are $F_\sigma$. Surprisingly,   these are precisely the countable subgroups of $\T$.  
  
\section{Characterized Subgroups of Compact Abelian Groups}

\subsection{Some known results on Borel Complexity}

It was proved by  B\' \i r\' o, Deshouillers and S\'os \cite{BDS}	that all countable subgroups of $\T$ are characterized. 
These authors conjectured that this fact could be extended to all 
compact metrizable abelian groups, although they gave no rigorous definition of a characterized
subgroup of a compact abelian group. The relevant definition 	\ref{DefChar}  was provided later only in \cite{DMT}, where cyclic subgroups of 
some compact metrizable abelian groups were proved to be characterized. 
The fact conjectured by B\' \i r\' o, Deshouillers and S\'os was proved to be true: 
	
\begin{teo}[{\cite[Theorem 1.4]{DK},\cite[Theorem 3.1]{BSW}}]\label{ThmDK}   The countable subgroups of a compact metrizable abelian group are characterized.	\end{teo}

\begin{fact} \label{l21}
It was pointed out in \cite{VN} that metrizability is a necessary condition in the above theorem. Indeed, for a sequence $\vs$ of characters of a compact abelian groups $X$, the subgroup $K_\vs=\bigcap_{n\in\omega} \ker v_n$ is a closed $G_\delta$-subgroup of $X$ contained in $\sv(X)$. Hence, if the (countable) subgroup $\{0\}$ of $X$ is 
characterized, then $\{0\}$ is a $G_\delta$-subgroup of $X$, so $X$ is metrizable. 
\end{fact}

According to the next theorem, the $G_\delta$-subgroups of a compact abelian group $X$ are precisely the closed characterized subgroups of $X$. 

\begin{teo}[{\cite[Theorem A]{DG}}]\label{TGchar} Let $X$ be a compact abelian group.
\begin{itemize}
     	 \item[(a)] $\SG(X)\subseteq\{\text{closed subgroups of }X\}$.
     	\item[(b)] $\SG(X)\subseteq\CH(X)$; more precisely, a closed subgroup $H$ of $X$ is characterized if and only if $H$ is a $G_\delta$-subgroup. 
\end{itemize}	
\end{teo}
	
\begin{rem} Note that if $X$ is metric then the inclusion of item (a) of the theorem becomes an equality. In case $X = \T$, item (b) of the above theorem is an easy consequence of item (a), as the closed subgroups of $\T$ are precisely the finite ones and they have the form $\T[m]=\{x\in\T:mx=0\}$ for some $m\in\N$ ($\T[m]$ is trivially characterized by the constant sequence $\us=(m)_{n\in\N}$).
\end{rem}
		
For a sequence $\vs$ of characters of a compact abelian groups $X$ the subgroup $\sv(X)/K_\vs$ of $X/K_\vs$ is characterized. More generally one has: 
	
\begin{teo}[{\cite[Theorem B]{DG}}]\label{ThB} For a subgroup $H$ of a compact abelian group $X$ 
the following are equivalent:
\begin{itemize}
     	 \item[(a)] $H$ is characterized;
     	 \item[(b)] $H$ contains a closed (necessarily characterized) $G_\delta$-subgroup $K$ of $X$ such that $H/K$ is a characterized subgroup of the compact metrizable group $X/K$. 
	 \end{itemize}
	 \end{teo}
	 
The previous theorem, reduces the study of characterizable subgroups of compact abelian groups, to the study of these subgroups in compact metrizable abelian groups. The general case of Theorem \ref{ThB}, namely the case of an arbitrary abelian group $X$, is considered in \cite{Ga5}.

Following \cite{DG}, call a subgroup $H$ of a compact abelian group $X$ {\em countable modulo compact} (briefly, {\em CMC}) if $H$ has a compact subgroup $K$ such that $H/K$ is countable and $K$ is a $G_\delta$-set of $X$. Clearly, CMC subgroups are $F_\sigma$, but they are also characterized subgroups: 
  
\begin{cor}[{\cite[Corollary B1]{DG}}]\label{CDKB} Let $ X$  be a compact abelian group and  let  $H$  be a CMC subgroup of  $X$. 
Then $H$ is characterized. 
\end{cor}
	
This follows easily from Theorems \ref{ThmDK} and \ref{ThB}  (see Proposition \ref{new:prop} for a stronger result). The above corollary was inspired by the following theorem. The subgroups characterized in the following theorem are particular CMC subgroups (in fact they are \emph{countable torsion modulo a compact subgroup}.

\begin{teo}[{\cite[Theorem 1.5]{DK}}]
 Let $X$ be a compact abelian group and $\{F_n\}_{n\in\N}$ a family of closed subgroups of $X$ such that, $F_n\lneq F_{n+1}$ for all $n\in\N$. Then $H=\bigcup F_n\in\CH(X)$ if and only if there exists $m\in\N$ such that $X/F_m$ is metrizable and $|F_{n+1}:F_n|<\infty$ for all $n\ge m$.
\end{teo}

 One can prove that every characterized subgroup $H$ of $X$ is CMC if and only if $X$ has finite exponent \cite{DG}.

A natural question arising from Corollary \ref{CDKB} is the following: 

\begin{que}
When an $F_\sigma$-subgroup $H$ of a compact metrizable abelian group $X$ is characterized?
\end{que}

The next theorem sharpening Remark \ref{rem1} was proved in \cite{DG}: 

\begin{teo}\label{teoDG1} {\rm \cite{DG}} For every infinite compact abelian group $X$, the following inclusions hold:
\begin{equation}\label{*}
 \SG(X) \subsetneqq \CH(X) \subsetneqq \SFF(X)\ \  \mbox{ and } \  \  \SF(X) \not \subseteq  \CH(X).
\end{equation}
If in addition $X$ has finite exponent, then
\begin{equation}\label{dag}
 \CH(X)  \subseteq \SF(X).
\end{equation}
\end{teo}
Clearly, one has $\CH(X)  \subsetneqq \SF(X)$ in (\ref{dag}), due to the second part of (\ref{*}).

Gabriyelyan \cite{Ga2} proved that the implication in the final part of the above Theorem can be inverted:

\begin{teo}[\cite{Ga2}]\label{ThGa} $\CH(X) \subseteq  \SF(X)$ for a compact abelian group $X$ if and only if $X$ has finite exponent.\end{teo}

In other words, for every compact abelian group of infinite exponent, he produced a characterized subgroup that is not $F_\sigma$.
The first example of a characterized subgroup of $\T$ that  is not an $F_\sigma$-subgroup of $\T$ was given by Bukovsk\'y, Kholshevikova, Repick\'y \cite{BKR} (see Example \ref{Exa_NonFS} in \S \ref{secaseq}). 

The above discussion motivates the following general question 

\begin{que}
When uncountable characterized subgroups of the compact metrizable groups are $F_\sigma$? 
\end{que}

This issue will be discussed in \ref{FS}, while the case of characterized subgroups of $\T$ is discussed in more detail in \S \ref{Toro}. 

The proof of the above theorem uses a relevant  property of characterized subgroups, noticed first by Bir\'o \cite{B} and independently, 
by Gabriyelyan \cite{Ga1}, \cite{Ga2}. Namely they are Polishable, as we see in the next section.


\subsection{Polishability}\label{SPol}

Let us recall the notion of Polishable subgroup that plays a key role in the study of characterized subgroups. It was introduced in \cite{KL}.

\begin{defin}\label{DPol}
	A {\em Polishable subgroup} $H$ of a Polish group $G$ is a subgroup that satisfies one of the following equivalent conditions: 
	\begin{itemize}
		\item[(a)] there exists a Polish group topology $\tau$ on $H$ having the same Borel sets as $H$ when considered as a topological subgroup of $G$; 
		\item[(b)] there exists a continuous isomorphism from a Polish group $P$ to $H$; 
		\item[(c)] there exists a continuous surjective homomorphism from a Polish group $P$ onto $H$. 
	\end{itemize}
\end{defin}

The topology witnessing the polishability of a subgroup is unique \cite{S}.

Answering negatively the first named author's question on whether $\F(\T)$ is contained in $\CH(\T)$, Bir\'o proved in 2005 the more precise Theorem \ref{BiroTh}. Before stating the theorem, let us recall the definition of Kronecker set.

\begin{defin}
 A non empty compact subset $K$ of an infinite compact metrizable abelian group $X$ is called a \emph{Kronecker set}, if for every continuous function $f:K\to\T$ and $\varepsilon>0$ there exists a character $v\in\widehat{X}$ such that
 \begin{equation*}
  \max\left\{\lb f(x)-v(x)\rb:x\in K\right\}<\ep.
 \end{equation*}
\end{defin}

\begin{notation}
 If $X$ is a group and $A\subseteq X$, denote by $\langle A\rangle$ the subgroup of $X$ generated by $A$.
\end{notation}

\begin{teo}[{\cite[Theorem 2]{B}}]\label{BiroTh}
	If $K$ is an uncountable Kronecker set of $\T$, then $\langle K\rangle\in\SF(\T)\setminus\CH(\T)$.
\end{teo}        

Inspired by ideas of Aaronson and Nadkarni \cite{AN}, he proved that every characterized subgroup of $\T$ is Polishable, while the subgroup $\langle K\rangle$ generated by any uncountable Kronecker set $K$, is not polishable and hence not characterized. Unaware of this result, Gabriyelyan \cite{Ga1} generalized significantly this theorem in several directions.

\begin{teo}\label{Tpolgen}  Let $X$ be a compact metrizable abelian group. Then
\begin{itemize}
 \item[(a)] {\rm \cite[Theorem 1]{Ga1}} $\sv(X)$ is Polishable for every sequence $\vs$ of characters of $X$;
 \item[(b)] {\rm \cite[Theorem 2]{Ga1}} if $K$ is an uncountable Kronecker set in $X$, then $\langle K\rangle$ is not Polishable; in particular, $\langle K\rangle$ is not characterized.
\end{itemize}
\end{teo}

The topology witnessing the polishability of $\sv(X)$ is induced  by the following metric.

\begin{defin}
  Let $X$ be a compact metrizable abelian group, $\delta$ be a compatible invariant metric on $X$ and $\vs=(v_n)$ be a sequence of characters of $X$. Let $x,y\in X$ and
  
  \begin{equation*}
   \varrho_\vs(x,y)=\sup_{n\in\N}\{\delta(x,y),d(v_n(x),v_n(y))\}.
  \end{equation*}
\end{defin}

Denote by $\tau_{\varrho_\vs}$ the topology of $X$ generated by the metric $\varrho_\vs$. By the uniqueness of the Polish topology, the restriction 
$\tau_{\varrho_\us}\restriction_{\su(X)}$ does not depend on $\uuu$, in the sense that $\tau_{\varrho_\us}\restriction_{\su(X)}= \tau_{\varrho_\vs}\restriction_{\sv(X)}$
whenever $\su(X)= \sv(X)$. This is why we denote by $\tau(H)$ the unique Polish topology of a characterized subgroup $H$ of $X$. Let us note here that 
{\em only this restriction} $\tau(H)$  was considered in \cite{B,Ga1}. The topology $\tau_{\varrho_\vs}$  on {\em the whole} group $X$ appeared for the first time only in \cite{DG}.

The assignments $\us \mapsto \su(X)$ and $\us \mapsto \tau_{\varrho_\us}$ give rise to two natural equivalence relations between sequences of characters of $X$. 

\begin{defin}
For two sequences of characters $\uuu$ and $\vvv$ of a compact abelian group $X$ we write 
\begin{itemize}
  \item $\uuu \sim \vvv$, if $\su(X)= \sv(X)$ 
  \item $\uuu \approx \vvv$, if $\tau_{\varrho_\us}= \tau_{\varrho_\vs}$. 
\end{itemize}
\end{defin}

As mentioned above, $\uuu \sim \vvv$ always implies also $\tau_{\varrho_\us}\restriction_{\su(X)}= \tau_{\varrho_\vs}\restriction_{\sv(X)}$,
but we do not know if $\uuu \sim \vvv$ implies $\uuu \approx \vvv$ in general  (see Question \ref{Qtopocoinc}). We shall see below that the answer is positive if the group 
$\su(X)= \sv(X)$ is $F_\sigma$ (Theorem \ref{CorC1}).

Let $\tau$ be the compact metrizable topology on $X$, it is clear that $\tau_{\varrho_\vs}\supseteq\tau$ and hence  $\tau(H) \supseteq \tau\restriction_H$ for $H\in\CH(X)$. 

\begin{rem}\label{11Ott} 
It is important to mention that $\tau(H) = \tau\restriction_H$, if and only if $H$ is closed. Indeed, $\tau\restriction_H$ is Polish if and only if $H\in\G(X)$. By item (a) of Theorem \ref{TGchar} the class $\G(X)$ coincides with class of closed subgroups of $X$. Using the uniqueness of $\tau(H)$, we can conclude that $\tau(H) = \tau\restriction_H$ precisely when $H$ is closed. 
\end{rem}

 In the light of the above mentioned results of Bir\'o and Gabriyelyan, it makes sense to introduce also the notation $\sPol(X)$ for the collection of all Polishable subgroups of $X$. Clearly, Theorem \ref{Tpolgen} (a) can be written briefly as
\begin{equation}\label{Eq1}
	\CH(X)\subseteq\sPol(X).
\end{equation}
One can expect to have equality in (\ref{Eq1}) at least for some infinite compact metrizable abelian groups $X$. That this cannot
occur follows from a general result of Hjorth \cite{Hj}. Answering a question of Farah and Solecki \cite[Question 6.1]{FS}, he 
 established that an uncountable abelian Polish group contains Polishable subgroups of unbounded Borel complexity. According to Remark \ref{rem1}, none of these Polishable subgroups can be characterized. Therefore, 
equality in (\ref{Eq1})  fails for every uncountable abelian Polish group $X$. 

In \cite{Ga4,Ga1} Gabriyelyan produced  compact groups $X$ witnessing $\sPol(X)\nsubseteq\CH(X)$ using  a different idea, since the Polishable subgroup he produced have some additional properties, being among others $F_\sigma$-subgroups. In this way he established  for these particular groups the following stronger non-inclusion
\begin{equation}\label{Eq93}
\sPol(X)\cap\SF(X)\not \subseteq\CH(X). 
\end{equation}

\begin{que}[\cite{Ga6}]\label{QG61}
 Does (\ref{Eq93}) holds for all compact metrizable abelian groups? 
\end{que}

Recently Gabriyelyan proved in \cite{Ga6} that \emph{the non-inclusion (\ref{Eq93}) holds for every compact metrizable non-totally disconnected abelian group}.

   
\subsection{Further properties of the polish topology of a characterized subgroup}\label{FurtherProp}

 A subset $A$ of a topological abelian group $G$ is called {\it quasi-convex} if for every $g\in G\setminus A$ there exists $\chi\in \widehat{G}$ such that
\[
\chi(A)\subset \T_+,\,\,\,\text{but}\,\,\chi(g)\not\in \T_+ \, ,
\]
where  $\T_+ $ is  the image of the segment $[-\frac{1}{4},\frac{1}{4} ]$ with respect to the natural quotient map $\R\rightarrow \T$.  A topological group $(G,\tau)$ (as well as its topology $\tau$) is called {\it locally quasi-convex} if $\mathcal{N}(G,\tau)$ $G$ has a basis of neighborhoods of $0$ consisting of quasi-convex subsets of $(G,\tau)$. 

In connection to the inclusion (\ref{Eq1}), Gabriyelyan \cite{Ga1} added in item (a) of Theorem \ref{Tpolgen} one more property of the finer Polish topology of a characterized subgroup $H$ of a compact metrizable abelian group $X$; namely, it is also locally quasi-convex. Therefore, if $\sPol_{qc}(X)$ denotes the family of all subgroups of $X$ that admits a finer polish locally quasi-convex topology, he refined (\ref{Eq1}) with the following inclusion:

\begin{equation}\label{Eq1qc}
	\CH(X)\subseteq\sPol_{qc}(X).
\end{equation}

On the other hand, he produced in \cite{Ga1} a compact metrizable abelian group $X$ with a locally quasi-convex Polishable $F_\sigma$-subgroup that is not characterized. Hence the following non-inclusion that improves (\ref{Eq93}) holds:

\begin{equation}\label{Eq93qc}
\sPol_{qc}(X)\cap\SF(X)\not \subseteq\CH(X). 
\end{equation}

The group produced in \cite{Ga1} that witnesses (\ref{Eq93qc}) is $\T^\N$. It makes sense to pose the following question related to Question \ref{QG61}.

\begin{que}[\cite{Ga6}]\label{QG61qc}
 Does (\ref{Eq93qc}) holds for all compact metrizable abelian groups?
\end{que}

 In the light of Bir\'o's result it would be natural to ask whether (\ref{Eq93}) holds for $\T$, namely  whether
 $\sPol(\T)\cap\SF(\T)\subseteq\CH(\T)$. A negative answer to this question can be found in  \cite{Ga4}. 
 Indeed, the group $G_2$ defined there is Polishable and $F_\sigma$ (by  \cite[Proposition 1]{Ga4}).   On the other hand, $G_2$ is not locally quasi-convex (by \cite[Theorem 2]{Ga4}),   so $G_2$ is not characterized. So it makes sense to ask Question \ref{QG61qc} for $X=\T$.
\begin{que}[\cite{Ga6}]\label{QG61qcT}
 Does (\ref{Eq93qc}) holds for $\T$?
\end{que}

A natural class of groups having both properties (namely, being simultaneously Polish and locally quasi convex) is the class of second countable locally compact abelian groups. Therefore, it makes sense to consider also the subfamily $\sLC(X)$ of $\sPol_{qc}(X)$, consisting of those subgroups $H\in \sPol(X)$ whose Polish topology is {\em locally compact}. This is justified by the following result from \cite{N}, answering a question from \cite{Ga3}: 

\begin{teo}[\cite{N}]\label{negro} If $G$ is a second countable locally compact abelian group and $p: G \to X$ is a continuous injective homomorphism into a compact metrizable abelian group $X$, then  $p(X)$ is a characterized subgroup of $X$. 
\end{teo}

This theorem provides a useful and handy sufficient condition for a subgroup of a compact metrizable abelian group to be characterized: 

\begin{corollary}\label{corollary:negro}
$\sLC(X)\subseteq\CH(X)$ for every compact metrizable abelian group $X$. 
\end{corollary}
 
  Indeed, if  $H \in \sLC(X)$ and $G$ denotes the group $H$ equipped with the finer locally compact group topology, then the inclusion map $p: G \hookrightarrow X$ is a continuous injective homomorphism, so $H = p(G)$ is a characterized subgroup of $X$ by  Theorem \ref{negro}.  
  
The next chain of inclusions summarizes (\ref{Eq1}), (\ref{Eq1qc}) and the above corollary
  \begin{equation*}
   \sLC(X)\subseteq\CH(X)\subseteq\sPol_{qc}(X).
  \end{equation*}
Recall that a topological space is \emph{$\sigma$-compact} if it is a countable union of compact subspace. Note that if $H\in\sLC(X)$, then $H$ is $\sigma$-compact (see \cite{K}) and hence obviously an $F_\sigma$-subgroup. Therefore, the following more precise inclusion holds
\begin{equation}\label{eqdag}
 \sLC(X)\subseteq\CH(X)\cap\F(X).
\end{equation}

\begin{exa}[\cite{HH}] Obviously,  $\sLC(X)$ contains all countable subgroups as well as all clopen subgroups of a compact metrizable group $X$. 
Sometimes, $\sLC(X)$ may be limited to exactly those subgroups of $X$, as in the case of the group $\J_p$ of $p$-adic integers, i.e., 
 $$\sLC(\J_p)=\{\text{countable subgroups}\}\cup\{\text{clopen subgroups}\}.$$
\end{exa}

It is easy to check, that every CMC subgroup of $X$ is in $   \sLC(X)$: 

\begin{prop}\label{new:prop} The CMC subgroups of a compact metrizable abelian group $X$ belong to $\sLC(X)$, so they are characterized 
$F_\sigma$-subgroups of $X$. 
\end{prop}

Note that this proposition provides a new proof of Corollary \ref{CDKB}. 

The next theorem describes when the implication of Proposition \ref{new:prop} can be inverted. 

\begin{teo}\label{new:prop1}\cite{HH} For a compact metrizable abelian group $X$ the following are equivalent:  
 \begin{itemize}
  \item[(a)] there exist no continuous injective homomorphisms $\R \to X$; 
  \item[(b)] all subgroups in $\sLC(X)$ are CMC. 
 \end{itemize}
\end{teo}

Obviously, $X$ satisfies (a) whenever $X$ is totally disconnected. More precisely, $X$ satisfies (a) if and only if its arc-component is either trivial
or isomorphic to $\T$ (for more detail see \cite{HH}).  

\subsection{When characterized subgroups of the compact metrizable groups are $F_\sigma$?
}\label{FS}

According to Theorem \ref{ThGa}, the answer to the above question is \lq\lq always\rq\rq, when $X$ is of finite exponent. For groups $X$ of infinite exponent, 
the same theorem provides always some characterized subgroup of $X$ that is not $F_\sigma$. Our aim is to precisely describe when a 
characterized subgroup of $X$ is $F_\sigma$. To this end one can introduce the 
following topology, obtained by a modification of $\tau_{\varrho_\vs}$.

\begin{defin}\label{DefTestTopo} {\rm \cite{DI3}} Let $\vs$ be a sequence of characters on a compact metrizable abelian group $(X,\tau)$ and $\tau_{\varrho_\vs}^*$ be the topology on $X$ having as a filter of neighborhoods of $0$ in $X$ the family 
$$
\left\{W_n=\overline{B^{\varrho_\vs}_{1/n}(0)}:n\in\N\right\},
$$ 
where $B^{\varrho_\vs}_{1/n}(0)$ is the ball around 0 of radius ${1/n}$ with respect to the metric $\varrho_\vs$  and $\overline{B^{\varrho_\vs}_{1/n}(0)}$ denotes its closure in $(X,\tau)$.
\end{defin}

We can refer to $\tau_{\varrho_\vs}^*$ as the \emph{$\F$-test topology} with respect to the sequence $\vs$. Note that this topology is  metrizable since it has a countable local base. Moreover if $\tau$ is the original compact topology on $X$, the following inclusions hold.

\begin{equation*}
 \tau\subseteq\tau_{\varrho_\vs}^*\subseteq\tau_{\varrho_\vs}.
\end{equation*}

Clearly, $\tau_{\varrho_\vs}$ is discrete if and only if $\tau_{\varrho_\vs}^*$ is discrete. 

The following theorem, providing a simple criterion for the solution of Problem \ref{probBor}, justifies the term test-topology for $\tau_{\varrho_\vs}^*$.

\begin{teo}[{\cite[Theorem A]{DI3}}]\label{ThC}
  Let $X$ be a metrizable compact abelian group and $\vs\in \widehat{X}^\N$. Then 
  \begin{equation*}
   \sv(X)\in\SF(X)\Longleftrightarrow\sv(X)\in\tau_{\varrho_\vs}^*.
  \end{equation*}
\end{teo}

An important consequence of the above theorem is the independence of the topology $\tau_{\varrho_\vs}$ on the choice of the characterizing sequence whenever $\sv(X)$ is an $F_\sigma$-subgroup. In such a case $(X,\tau_{\varrho_\vs})$ is Polish precisely when $\sv(X)$ is an open subgroup of $X$: 

\begin{teo}[{\cite[Corollary A1]{DI3}}]\label{CorC1}
 Let $X$ be a metrizable compact abelian group and $\vs$ be a sequence of characters such that $\sv(X)\in\F(X)$. 
 \begin{itemize}
  \item[(a)] If $\us$ is a sequence of characters such that $\vs\sim\us$, then $\vs\approx\us$.
  \item[(b)] The following are equivalent:  
  \begin{itemize}
      \item[(b$_1$)] $(X,\tau_{\varrho_\vs})$ is Polish; 
      \item[(b$_2$)] $(X,\tau_{\varrho_\vs})$ is separable; 
      \item[(b$_3$)] $\sv(X)$ is a closed finite-index (so, open) subgroup of $X$.
  \end{itemize}
\end{itemize}
\end{teo}

It is not known if item (a) of the above theorem remains valid without the assumption that $H$ is an $F_\sigma$-subgroup (see Question \ref{Qtopocoinc}). 

  \begin{cor}\label{new:cor} Let $X$ be a compact metrizable abelian group. If $H \in \sLC(X)$, then for every pair of characterizing sequences $\us$ and $\vs$ for $H$ the topologies $\tau_{\varrho_\us}$ and $\tau_{\varrho_\vs}$ coincide. 
\end{cor}

\begin{proof}
Follows from Theorem \ref{CorC1} and (\ref{eqdag}). 
\end{proof}

From Theorem \ref{ThC} we deduce now that $\tau_{\varrho_\vs}$ is discrete on the whole $X$ whenever $\sv(X)$ is countable. 

\begin{teo}[{\cite[Theorem B]{DI3}}]\label{CorC2}
	 Let $X$ be a metrizable compact abelian group and $\vs\in \widehat{X}^\N$. Then the following are equivalent: 
\begin{itemize}
  \item[(a)]  $\sv(X)$ is countable;
  \item[(b)] $|\sv(X)|<\cc$;
  \item[(c)]  $\tau_{\varrho_\vs}$ is discrete;
  \item[(d)] $\tau_{\varrho_\vs}^*$ is discrete;
  \item[(e)]  $\tau_{\varrho_\vs}^*\restriction_{\sv(X)}$ is discrete;
  \item[(f)] $\tau({\sv(X)})$ is discrete.
\end{itemize}
\end{teo}

\begin{proof}
(a)$\Leftrightarrow$(b). It is a consequence of Remark \ref{rem1}

(a)$\Rightarrow(c)$. Indeed if $\sv(X)$ is countable, then it is obviously in $\F(X)$. Hence, $\sv(X)$ is $\tau_{\varrho_\vs}$-open, by Theorem \ref{ThC}. Moreover $\tau_{\varrho_\vs}$ is discrete when restricted to $\sv(X)$ since the topology that witnesses the polishability of $\sv(X)$ is unique. Thus, $\sv(X)$ is both $\tau_{\varrho_\vs}$-open and $\tau_{\varrho_\vs}$-discrete. Consequently,  $\tau_{\varrho_\vs}$ is discrete.

(c)$\Rightarrow$(d) follows from the definition of $\tau_{\varrho_\vs}^*$, while (d)$\Rightarrow$(e) is obvious.

(e)$\Rightarrow$(f) follows from the fact that $\tau({\sv(X)})=\tau_{\varrho_\vs}\restriction_{\sv(X)}$ is finer then $\tau_{\varrho_\vs}^*\restriction_{\sv(X)}$.

(f)$\Rightarrow$(a) follows from the fact that $(\sv(X),\tau({\sv(X)}))$ is separable.
 \end{proof}

This theorem shows, that $\uuu\approx \vvv$ does not imply $\uuu \sim  \vvv$. Indeed, if $H$ and $H'$ two distinct countable subgroups, then 
they are characterized by Theorem \ref{ThmDK}, say $H= \tau_{\varrho_\us}$ and $H'=\tau_{\varrho_\vs}$. 
According to Theorem \ref{CorC2}, $\tau_{\varrho_\us}=\tau_{\varrho_\vs}$ are discrete. So $\uuu\approx \vvv$, while $H' \ne H$, i.e., $\uuu \not\sim  \vvv$.

\section{Characterized subgroups of the Circle Group}\label{Toro}
\subsection{Preliminaries}
\begin{defin}
For   a topological abelian group $X$  and a sequence $\us=(u_n)_{n\in\N}$  of integers, let 
$$
\tu(X)=\left\{x\in X:u_nx\to0 \text{ in } X \right\}.
$$
\end{defin}

This subset is actually a subgroup of $X$ and has been studied in \cite{Arm,D,BDMW2}. Here we consider only the case when $X = \T$.  
Since $\widehat{\T}=\Z$, the notion defined above coincides (in the case of $X=\T$) with the notion of characterized subgroup of $\T$. In fact, 
$\tTu=\sv(\T)$ where for all $n\in\N$ the character $v_n$ of $\T$  is defined by $v_n:x\mapsto u_nx$. 
In other words, $H\le\T$ is characterized if and only if there exists $\us\in\Z^{\N}$ such that $H=\tTu$.

The above described circumstance provides additional tools to use in the study of characterized subgroups of $\T$, 
namely the specific properties of $\Z$. 

As a first easy reduction, one can note that if $\us$ does not have any constant subsequences one can find  a strictly increasing sequence $\us^*$ of non-negative integers such that $\us^*\sim \uuu$ (see \cite[Proposition 2.5]{BDMW1}). 
 
 Another  powerful tool to use in the study of characterized subgroups of $\T$ is the sequence of ratios $(q_n):=(\frac{u_n}{u_{n-1}})$ for a sequence $\us$ of non-zero integers. Furthermore, 
let 
$$
q^\us: = \limsup_n\frac{u_{n+1}}{u_n}\mbox{  and  }q_\us: = \liminf_n\frac{u_{n+1}}{u_n}. 
$$
Note that the sequence of ratios $(q_n)$ is bounded precisely when $q^\us$ is finite (one can also say that $\us$ is \emph{$q$-bounded}), while 
$(q_n)$ converges to infinity precisely when $q_\us = \infty$ (i.e. $\us$ is \emph{$q$-divergent}). 
In these terms, Eggleston \cite{Egg} proved the following remarkable theorem  (see also \cite{BDMW1}).

\begin{teo}[Eggleston]\label{TEgg} Let $\us=(u_n)$ a sequence of integers.
  \begin{itemize}
      \item[(a)] If $q^\us < \infty$ (i.e., the sequence of ratios $(q_n)$ is bounded), then $\tTu$ is countable.
      \item[(b)] If $q_\us = \infty$ (i.e., the sequence of ratios $(q_n)$ converges to infinity), then $\tTu$ is uncountable.
  \end{itemize}
\end{teo}

One can show that the implications of Eggleston Theorem \ref{TEgg} cannot be inverted:  

\begin{rem}[{\cite[Remark 3.5]{BDMW1}}]\label{B-M-W} If $H$ is an infinite characterized subgroup of $\T$, with characterizing sequence $\us\in\Z^\N$, then for every 
$m\in\N$ one can find a strictly increasing sequence $\us^*\sim \uuu$ such that $q^{\us^*}=\infty$ (i.e., 
$\limsup_n\frac{u_{n+1}^*}{u_n^*}=\infty$) and $q_{\us^*}=m$
(i.e., $\liminf_n\frac{u_{n+1}^*}{u_n^*}=m$). 
\end{rem}

This remark shows that the properties of having bounded ratios, or having ratios converging to $\infty$ are not $\sim$-invariant. Actually, if one takes a sequence $\uuu$ with bounded ratios, and then a sequence $\us^*\sim \uuu$ as in the remark, then one will have also $\us^*\approx \uuu$, according to Theorem \ref{CorC1}, while $\uuu^*$ will fail to have the property of having bounded ratios. 

On the other hand the following remark,  answering a question posed in \cite{BDMW1}, is a sort of inversion of the first Eggleston implication.

\begin{rem}[{\cite[Theorem 4.1]{BSW}}]
 Every countable subgroup of the circle group, admits a $q$-bounded characterizing sequence.
\end{rem}

If $(q_n)$ is bounded, then $\tTu$ is countable and hence obviously $F_\sigma$. One can study the relations between these three properties, that are not equivalent  in general (see Theorem \ref{ThE} and remarks \ref{Ranotd} and \ref{Rbnotd}). To this end we distinguish two cases, the first one is the case of a general sequence of integers. The second one is the case of a particular kind of sequences of integers, namely the arithmetic sequences (see Definition \ref{Def_a_seq}). In this case the three 
properties mentioned above are equivalent. 

\subsection{General Sequences}  
 
 According to Theorem \ref{CorC2}, when $\sv(X)$ is countable, then the topology $\tau_{\varrho_\vs}$ is discrete on $X$, where $X$ is a compact metrizable abelian group. Hence, for $X=\T$ and a sequence of integers $\us$, having a  bounded sequence of ratios $(q_n)$, the topology $\tau_{\varrho_\vs}$ is discrete on $\T$
 in view of Theorem \ref{TEgg} (a). However, one can say something more precise than Theorem \ref{CorC2}: 

\begin{teo}[{\cite[Theorem C]{DI3}}]\label{ThA}
Let $\us$ be a strictly increasing sequence of positive integers such that its sequence of ratios $(q_n)$ is bounded. In that case $\tau_{\varrho_\us}$ is discrete. 

In particular if $0<C\in\R$ and $q_n\le C$ for all $n\in\N$, then $B_{\frac{1}{2C}}^{\varrho_\us}(0) = \{0\}$ in $\T$.
\end{teo}

The proof of the next theorem uses a property of $\T$ first noticed by Hewitt \cite{H}, and then extended to all subgroups of $\T$ (and elsewhere) 
in \cite{HH}: {\em if $H$ is a subgroup of $\T$ and $\tau$ is a strictly finer locally compact group topology on $H$, then $\tau$ is discrete.} Hence one can add to Theorem \ref{CorC2} one more equivalent condition in the case $X=\T$, given by the next theorem.

\begin{teo}[{\cite[Theorem 2.1]{HH}}]\label{ThD} 
Let $H\lneq\T$. Then $H\in\sLC(\T)$, if and only if $H$ is countable.
\end{teo}
Let us recall that  $\sLC(\T)\subset\CH(\T)$ (by Corollary \ref{corollary:negro}) and $H=\tTu$ coincides with $\T$ if and only if $\us$ is definitely $0$ (see \cite[Example 2.8]{BDMW1}). 

It turns out that this theorem cannot be proved in more general situation, i.e., it is specific for the group $\T$. 

\begin{teo}[{\cite[Theorem 2.2]{HH}}]\label{ThDD} 
 Let $X$ be an infinite compact abelian group. Then the following are equivalent: 
\begin{itemize}
    \item[(a)]  $X \cong \T$; 
    \item[(b)]  whenever $H\in\sLC(X)$  and  $H \ne X$, then $H$ is countable.
    \end{itemize} 
\end{teo}

\subsection{Arithmetic Sequences}\label{secaseq}

In the case of $\T$ one has the following two examples of characterized non-$F_\sigma$-subgroups: 
 
\begin{example}\label{Exa_NonFS}
\begin{itemize}
    \item[(a)]   For the sequence $\us = (2^{2^n})$ Bukovsk\'y, Kholshevikova, Repick\'y \cite{BKR} proved that  
the characterized subgroup $\tTu$ of $\T$  is not an $F_\sigma$-set.
    \item[(b)]  For the sequence $\us = (n!)$ Gabriyelyan \cite{Ga2} proved that the subgroup $\tTu$ of $\T$  is not an $F_\sigma$-subgroup of $\T$. 
\end{itemize}
\end{example}

The obvious common features between  (a) and (b) are $u_n | u_{n+1}$ and $q_n= \frac{u_{n+1}}{u_n} \to \infty$. 
We are going to use the first one to define the following new notion.

\begin{defin}\label{Def_a_seq} 
  A strictly increasing sequence of positive integers $\us = (u_n)$  is called {\em arithmetic} (or briefly, an {\em a-sequence}) if $u_n|u_{n+1}$ for all $n\in\N$. 
\end{defin}

\begin{notation}
  Let $\PP$ be the set of all prime numbers. If $p\in\PP$, then $\Z(p^\infty)$ is the Pr\"ufer $p$-group and $\Z(p^n)$ the cyclic group of order $p^n$ where $n\in\N$.
\end{notation}

Note that in this case $(q_n)$ is a sequence of integers. The problem of the description of the structure of $\tTu$ in the case of arithmetic sequences has been raised in  \cite[Chap.4]{DPS}, where some partial results can be found. Much earlier, Armacost \cite{Arm} considered two special kinds of  
a-sequences, $u_n = n!$ and $u_n = p^n$ for some $p \in \PP$. The subgroup characterized by the former type is $\Z(p^\infty)$, as shown in \cite{Arm}. 
The subgroup characterized by $u_n = n!$ was denoted by $\T!$, its description was left as an open problem in \cite{Arm}, 
resolved independently by Borel \cite{Bo2} and \cite{DPS}. Further results on subgroups of $\T$ characterized by a-sequences can be found in \cite{DdS,DI}.
A complete description of this class of subgroups of $\T$ is given in  \cite{DI}.

\begin{exa}\label{RemBDMW}(\cite{BDMW}) If $\us$ is an a-sequence and $(q_n)$ is bounded, then 
$$
\tTu\cong\left(\bigoplus_{i=1}^s\Z(p_i^{k_i})\right) \oplus\left(\bigoplus_{j=1}^r\Z(t_j^\infty)\right);
$$ 
where $p_i\in\PP$, $k_i=|\{n:p_i|q_n\}|<\infty$ and $t_j$ are primes that divide infinitely many $q_n$. It is easy to see that 
the subgroups of this form of $\T$ are countably many. 
\end{exa}

 In the case of arithmetic sequences one can add some further equivalent conditions to the conditions of Theorem \ref{CorC2} as follows.

\begin{teo}[{\cite[Theorem E]{DI3}}]\label{ThE}
 
The following are equivalent for an a-sequence $\us$ in $\Z$:
\begin{itemize}
    \item[(a)] $\tTu\le\Q/\Z$;
    \item[(b)] $(q_n)$ is bounded;
    \item[(c)] $\tTu$ is countable;
    \item[(d)] $\tTu\in \SF(\T)$;
    \item[(e)] $\tTu\in\tau_{\varrho_\us}$.
\end{itemize}
\end{teo}

The implications (b)$\Rightarrow$(c)$\Rightarrow$(d) are immediate (the first one due to Theorem \ref{TEgg} (a)), while the equivalence between (a) and (b) is proved in  \cite{DdS,DI}. The implication (d)$\Rightarrow$(e) is a consequence of the Theorem \ref{ThC} and the fact that $\tau_\varrho^*\subseteq\tau_\varrho$. The last implication (e)$\Rightarrow$(b) is proved in \cite{DI3}.

From the equivalence of (a) and (b) in the above theorem and Example \ref{RemBDMW}, we deduce that {\em 
only countably many subgroups of $\Q/\Z$ can be characterized by means of an a-sequence $\us$ in $\Z$}.  

Let us see that some of the implications of Theorem \ref{ThE} do not hold for a general $\us\in\Z^\N$.

\begin{rem}[(b)$\nRightarrow$(a)$\nLeftarrow$(c)]\label{Rbnotd}
If $\us$ is the Fibonacci's sequence, then $(q_n)$ is bounded, so $\tTu$ is countable (by virtue of Theorem \ref{TEgg}),
 i.e. (b) and (c) of Theorem \ref{ThE} holds. Indeed $u_n=u_{n-1}+u_{n-2}$ for all $n>1$ and $u_0=u_1=1$. Hence, $q_n=\frac{u_n}{u_{n-1}} = 
 1 + \frac{u_{n-2}}{u_{n-1}} \le2$ for all $n\in\N$.

On the other hand, $\tTu$ is the infinite cyclic group generated by the fractional part of the golden ratio
 \cite{La,BDMW,BDS}.  In particular, $\tTu$ is not torsion. 
\end{rem} 

The following remark shows that (a)$\nRightarrow$(b)$\nLeftarrow$(c) in Theorem \ref{ThE} in case $\us$ is not an a-sequence.

\begin{rem}[(a)$\nRightarrow$(b)$\nLeftarrow$(c)]\label{Ranotd}
Take any infinite subgroup $H$ of $\Q/\Z$. By Remark \ref{B-M-W}, 
$H$ has a characterizing sequence such that its sequence of ratios is unbounded. This proves the non-implication (a)$\nRightarrow$(b) in Theorem \ref{ThE}.
As (a) implies (c), this witnesses also the non-implication  (c)$\nRightarrow$(b)
\end{rem}

We do not know if the remaining implication (d)$\Rightarrow$(c) is true in the general case (see Question \ref{Q1}).

The following diagram summarizes Theorem \ref{ThE} and remarks \ref{Rbnotd} and  \ref{Ranotd}, where the question mark denotes the unknown implication and the slashes denote the failing implications.

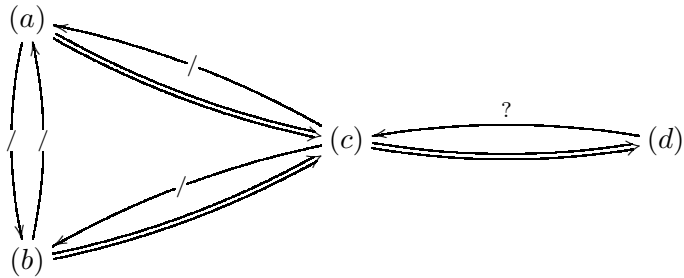
\[
{\xymatrix@!0@C4.2cm@R=1.6cm{ 
(a) \ar@{=>}@/_/[rd] \ar|{/}@/_/[dd] & & \\
  &(c)\ar|{/}@/_/[ul] \ar|{/}@/_/[dl]\ar@{=>}@/_/[r]	&(d) \ar@/_/_{?}[l]\\
(b) \ar|{/}@/_/[uu] \ar@{=>}@/_/[ur]& &
} }
\]



\section{Questions and Remarks}

We start with some questions and remarks concerning characterizing sequences. 

\begin{rem} If $\ACH(\T)$ denotes the set of  all subgroups of $\T$ characterized by an a-sequence, then
\begin{equation*}
	\{\text{closed subgroups}\}=\SG(\T)\subsetneq\ACH(\T)\cap\SF(\T).
\end{equation*}
An example witnessing the above proper inclusion is $\tTu\in\SF(\T)\setminus\SG(\T)$ where $\us=(p^n)$. More generally all  countably infinite $\tTu\in \ACH(\T)$ witness that proper inclusion, where $\us$ is an a-sequence.
\end{rem}

To give an exhaustive description of $\CH(\T)$ and $\ACH(\T)$, in terms of Problem \ref{probBor}, it is needed to establish if there exists a sequence
$\us$ such that $\tTu\notin\SGG(\T)$. This kind of sets is called $F_{\sigma\delta}$-complete (or $\Pi_3^0$-complete using the Descriptive Set-Theoretic terminology, see \cite{K}), that is,  in plain words,
 a set among the most complex in $\FF(\T)$. Hence question of whether $\CH(\T)\subseteq \SGG(\T)$ remains open: 
\begin{que}
\begin{itemize}
    \item[(a)] (\cite{Ga2}) Does there exists a sequence of integers $\us$ such that $\tTu\notin\SGG(\T)$?
    \item[(b)] What about a-sequences $\us$?
\end{itemize}
\end{que}

\begin{que}	
Does there exists an explicit method to decide whether $\tTu\in\SGG(\T)$ for a sequence of integers  $\us$? What about a-sequences $\us$?
\end{que}

Since $t(\T) = \Q/\Z$ is countable, all ($\mathfrak c$ many) torsion subgroups of $\T$ can be characterized, by Theorem \ref{ThmDK}. 
On the other hand, not all the torsion subgroups of $\T$ can be characterized by an a-sequence.  Indeed, by Example \ref{RemBDMW} and Theorem \ref{ThE}, all torsion subgroups of $\T$ of the form $\bigoplus_{i\in I}\Z(p_i^{k_i})$  with $k_i\in\N\cup\{\infty\}$ for all $i\in I, p_i\in\PP$,  and $p_i\neq p_j$ for  $i\neq j$ and $\left|I\right|=\aleph_0$ cannot be characterized by an a-sequence. Hence, as far as only torsion subgroups of $\T$ are concerned, one can say that
$\ACH(\T)$ is much smaller with respect to $\CH(\T)$ (as $\ACH(\T)$ contains only countably many torsion subgroups, while 
$\CH(\T)$ contains $\mathfrak c$ many torsion subgroups of $\T$). 

 According to   Theorem \ref{ThDD},  every compact metrizable abelian group $X$ non-isomorphic to the circle group, has a proper uncountable 
 $F_\sigma$-subgroup $H$ that is characterized. Indeed, such an $H$ can be chosen to admit a finer locally compact Polish topology, hence $H$ is characterized by Corollary \ref{corollary:negro} and $F_\sigma$ by (\ref{eqdag}). On the other hand, for $X=\T$ this matter remains unclear. This leaves open 
 the following question, related to the implication  (d)$\Rightarrow$(c) in Theorem \ref{ThE}.

\begin{que}
\label{Q1} If $\us\in\Z^\N$ is not definitely $0$ and $\tTu\in\SF(\T)$, must $\tTu$ be necessarily countable?
\end{que}

The final questions concern Polishable subgroups. 
 
 According to  Theorem \ref{CorC1}, the topology $\tau_{\varrho_\vs}$ does not depend  on the choice of $\vs$ in case $\sv(X)$ is an $F_\sigma$-subgroup of $X$, however the questions remain open in the general case: 

\begin{que}\label{Qtopocoinc}
	If $\sv(X)=\su(X)\notin\F(X)$, is $\tau_{\varrho_\vs}=\tau_{\varrho_\us}$ in the whole $X$?
\end{que}

In Theorem \ref{CorC1} we described when $(X,\tau_{\varrho_\vs})$ is Polish, provided $\sv(X)$ is an $F_\sigma$-subgroup of $X$. 

\begin{que}\label{QtopoPol}
Describe when $(X,\tau_{\varrho_\vs})$ is Polish in case $\sv(X)\notin\F(X)$.
\end{que}

 \section*{Acknowledgements}
 It is a pleasure to thank S. S. Gabriyelyan for his helpful remarks and suggestions that helped us to improve our paper. 



\begin{thebibliography}{1000MM.}

\bibitem[AN]{AN} J. Aaronson and M. Nadkarni,  $L_\infty$  eigenvalues  and $L_2$  spectra of no-singular transformations, Proc. London Math. Soc. {\bf 55} (3) (1987) pp. 538--570.


\bibitem[Arm]{Arm} D. Armacost, {\it The structure of locally compact abelian groups}, Monographs and Textbooks in Pure and Applied Mathematics {\bf 68}, Marcel Dekker Inc., New York, (1981).

\bibitem[Bi]{B} A. Bir\' o,  \emph{Characterizations of groups generated by Kronecker sets}, Journal de th\' eorie des nombres de Bordeaux {\bf 19} (3) (2007) pp. 567--582.

\bibitem[BDMW1]{BDMW1} G. Barbieri, D. Dikranjan, C. Milan, H. Weber, {\em Answer to Raczkowki's questions on convergent sequences of integers}, Topology Appl. {\bf 132}  (2003) pp. 89-101.

\bibitem[BDMW2]{BDMW2} G. Barbieri, D. Dikranjan, C. Milan, H. Weber, {\em $\mathfrak t$-dense subgroups of topological abelian groups}, 
Questions \& Answers Gen. Topology {\bf 24} (2) (2006) pp. 99--118. 

\bibitem[BDMW3]{BDMW} G. Barbieri, D. Dikranjan, C. Milan, H. Weber, {\em Topological torsion related to some recursive sequences of integer}, Math. Nachrichten {\bf 281} (7) (2008) pp. 930-950.
 
\bibitem[BDS]{BDS} A.  B\' \i r\' o, J. -M. Deshouillers and V.T. S\'os, {\em Good approximation and characterization of subgroups of $\R/\Z$}, Studia Sci. Math. Hungar. {\bf 38} (2001) pp. 97--113. 


\bibitem[BSW]{BSW} M. Beiglb\" ock, C. Steineder, and R. Winkler, \emph{Sequences and filters of characters characterizing subgroups of compact abelian groups}, Topology Appl. {\bf 153} (11) (2006) pp. 1682--1695.
 
 \bibitem[Bo1]{Bo1} J. P. Borel,  {\em Sous-groupes de $\R$ li\'es \` a r\'epartition modulo $1$ de suites} Ann. Fac. Sci. Toulouse Math. {\bf 5} (3-4) (1983) pp. 217--235.
 
 \bibitem[Bo2]{Bo2}   J.-P. Borel, {\em Sur certains sous-groupes de $\R$ li\' es \' a la suite des factorielles}, Colloq. Math. {\bf 62} (1)  (1991) pp. 21--30. 
 
 \bibitem[BKR]{BKR}  L. Bukovsk\'y, N. N.  Kholshevikova, M.  Repick\'y, \emph{Thin  sets  in  harmonic  analysis  and  infinite combinatorics},  Real  Anal.  Exchange  {\bf 20}  (1994/1995)  pp. 454--509.

\bibitem[D1]{D} D. Dikranjan,   {\em  Topologically torsion elements of topological groups},  Topology Proc.  {\bf 26} (2001-2002) pp. 505--532.

 \bibitem[D2]{VN} D. Dikranjan,   {\em  Closure operators in topological groups related to von Neumann's kernel}, Topology Appl. {\bf 153} (2006) pp. 1930--1955.
 
\bibitem[DdS]{DdS} D. Dikranjan  and R. Di Santo,  {\it Answer to Armacost's quest on topologically torsion elements of the circle group}, Comm. Algebra {\bf 32} (2004) pp. 133--146.

\bibitem[DI1]{DI} D. Dikranjan  and D. Impieri, {\em Topologically torsion elements of the circle group}, Comm. Algebra {\bf 42} (2014) pp. 600--614.  

\bibitem[DI2]{DI3} D. Dikranjan  and D. Impieri, {\em On the Borel complexity of characterized subgroups}, preprint. 

 \bibitem[DI3]{HH} D. Dikranjan and D. Impieri,   {\em  On Hewitt groups and finer locally compact group topologies}, preprint. 
 
 \bibitem[DG]{DG} D. Dikranjan and {S.~S. Gabriyelyan},  {\em On characterized subgroups of compact abelian groups}, Topology Appl., {\bf 160} (18) (2013) pp. 2427--2442. 

\bibitem[DGT]{DGT} D. Dikranjan, {S.~S. Gabriyelyan}  and V. Tarieladze,  {\em Characterizing sequences for precompact group topologies}, J. Math. Analysis and Applications, {\bf 412} (1) (2014) pp. 505--519. 

\bibitem[DK]{DK} D. Dikranjan and K. Kunen,   \emph{Characterizing countable subgroups of compact abelian groups}, J. Pure Appl. Algebra {\bf 208} (2007) pp. 285--291.

\bibitem[DMT]{DMT}  D. Dikranjan, C. Milan and A. Tonolo, {\em A characterization of the maximally almost periodic abelian groups}, J. Pure Appl. Algebra {\bf 197} (1--3) (2005) pp. 23--41. 

\bibitem[DPS]{DPS} D. Dikranjan, Iv. Prodanov and L. Stoyanov, {\it Topological Groups: Characters,  Dualities  and  
Minimal Group Topologies},  Pure and Applied Mathematics,  Vol. {\bf 130}, Marcel Dekker  Inc., New York-Basel, (1989).

\bibitem[Egg]{Egg} H. G. Eggleston {\em Sets of fractional dimensions which occur in some problems of number theory}, Proc. London Math. Soc. {\bf 54} (2) (1952) pp. 42--93.

\bibitem[FS]{FS} I. Farah, S. Solecki, \emph{Borel subgroups of Polish groups}, Adv. Math. {\bf 199} (2) (2006) pp. 499--541.

\bibitem[G1]{Ga4} {S.~S. Gabriyelyan}, \emph{Group of quasi-invariance and the Pontryagin duality}, Topology Appl. {\bf 157} (2010) pp. 2786--2802.

\bibitem[G2]{Ga1} {S.~S. Gabriyelyan}, \emph{On $T$-sequence and characterized subgroups}, Topology Appl. {\bf 157} (2010) pp. 2834--2843.

\bibitem[G3]{Ga3} {S.~S. Gabriyelyan}, \emph{Characterizable groups: some results and open questions}, Topology Appl. {\bf 159} (2012) pp. 2378--2391.

\bibitem[G4]{Ga2} {S.~S. Gabriyelyan}, \emph{On $T$-characterized subgroups of compact Abelian groups}, submitted.

\bibitem[G5]{Ga5} {S.~S. Gabriyelyan}, \emph{On characterized subgroups of Abelian topological groups $X$ and the group of all $X$-valued null sequences}, Comment. Math. Univ. Caroli. to appear.

\bibitem[G6]{Ga6} {S.~S. Gabriyelyan} \emph{On Borel complexity of characterized subgroups}, Work in progress.

\bibitem[H]{H} E. Hewitt,  {\em A remark on characters of locally compact Abelian groups}, Fund. Math. {\bf 53} (1963) pp. 55--64. 

\bibitem[Hj]{Hj} G. Hjorth, {\em Subgroups of abelian Polish groups},  Set theory, Trends Math., BirkhŠuser, Basel, (2006) pp. 297--308. 

\bibitem[K]{K} A.S. Kechris, {\em Classical descriptive set theory}, Graduate Texts in Mathematics  {\bf 156}, Springer, (1995).

\bibitem[KL]{KL} A.S. Kechris and Louveau , {\em The classification of the structure of hypersmooth  Borel equivalence relations}, J. Amer. Math. Soc. {\bf 10} (1997) pp. 215-242. 

\bibitem[KK]{KK} K. Kuratowski, \emph{Topology, Vol. 1}, Academic Press/PWN , New York-London-Warszawa, (1966).

\bibitem[L]{La} G. Larcher, {\em A convergence problem connected with continued fractions}, Proc. Amer. Math. Soc {\bf 103} (3) (1988) pp. 718--722.

\bibitem[N]{N} G. Negro, \emph{Polish LCA groups are strongly characterizable}, Topology and its Applications {\bf 162} (2014) pp. 66--75.

\bibitem[S]{S} S. Solecki, \emph{Polish group topologies}, in Sets and Proofs, London Mathematical Society Lecture Notes Series {\bf 258}, Cambridge University Press (1999), pp. 339--364
\end{thebibliography}
\end{document}